\documentclass[12pt,reqno]{article}
\usepackage{caption}
\usepackage{diagbox}
\usepackage{verbatim}
\usepackage[usenames]{color}
\usepackage{amssymb}
\usepackage{graphicx}
\usepackage{amscd}

\usepackage{enumerate}

\usepackage{tikz}
\usetikzlibrary{decorations.text}

\usepackage[colorlinks=true,
linkcolor=webgreen,
filecolor=webbrown,
citecolor=webgreen]{hyperref}

\definecolor{webgreen}{rgb}{0,.5,0}
\definecolor{webbrown}{rgb}{.6,0,0}

\usepackage{color}
\usepackage{fullpage}
\usepackage{float}

\usepackage{graphics,amsmath,amssymb}
\usepackage{amsthm}
\usepackage{amsfonts}
\usepackage{latexsym}
\usepackage{epsf}

\DeclareMathOperator{\ham}{ham}

\usepackage{breakurl}

\def\mycommand{\setlength{\abovedisplayskip}{8pt}%
\setlength{\belowdisplayskip}{8pt}%
\setlength{\abovedisplayshortskip}{8pt}%
\setlength{\belowdisplayshortskip}{8pt}}

\let\oldselectfont\selectfont
\def\selectfont{\oldselectfont\mycommand}

\mycommand

\begin{document}

\theoremstyle{plain}
\newtheorem{theorem}{Theorem}
\newtheorem{corollary}[theorem]{Corollary}
\newtheorem{lemma}[theorem]{Lemma}
\newtheorem{proposition}[theorem]{Proposition}

\theoremstyle{definition}
\newtheorem{definition}[theorem]{Definition}
\newtheorem{example}[theorem]{Example}
\newtheorem{conjecture}[theorem]{Conjecture}

\theoremstyle{remark}
\newtheorem{remark}[theorem]{Remark}

\author{
Daniel Gabric\\
School of Computer Science\\
University of Waterloo\\
Waterloo, Ontario N2L 3G1\\
Canada\\
\href{mailto:dgabric@uwaterloo.ca}{\tt dgabric@uwaterloo.ca}
}

\title{Words that almost commute}

\maketitle

\begin{abstract}
The \emph{Hamming distance} $\ham(u,v)$ between two equal-length words $u$, $v$ is the number of positions where $u$ and $v$ differ. The words $u$ and $v$ are said to be \emph{conjugates} if there exist non-empty words $x,y$ such that $u=xy$ and $v=yx$. The smallest value $\ham(xy,yx)$ can take on is $0$, when $x$ and $y$ commute. But, interestingly, the next smallest value $\ham(xy,yx)$ can take on is $2$ and not $1$. In this paper, we consider conjugates $u=xy$ and $v=yx$ where $\ham(xy,yx)=2$.  More specifically, we provide an efficient formula to count the number $h(n)$ of length-$n$ words $u=xy$ over a $k$-letter alphabet that have a conjugate $v=yx$ such that $\ham(xy,yx)=2$. We also provide efficient formulae for other quantities closely related to $h(n)$. Finally, we show that $h(n)$ grows erratically: cubically for $n$ prime, but exponentially for $n$ even.
\end{abstract}

\section{Introduction}
Let $\Sigma_k$ denote the alphabet $\{0,1,\ldots, k-1\}$. Let $u$ and $v$ be two words of equal length. The \emph{Hamming distance} $\ham(u,v)$ between $u$ and $v$ is defined to be the number of positions where $u$ and $v$ differ~\cite{Hamming:1950}. For example, $\ham({\tt four}, {\tt five}) = 3.$

A word $w$ is said to be a \emph{power} if it can be written as $w=z^i$ for some word $z$ where $i\geq 2$. Otherwise $w$ is said to be \emph{primitive}. For example, ${\tt hotshots} = ({\tt hots})^2$ is a power, but ${\tt hots}$ is primitive.  The words $u$ and $v$ are said to be \emph{conjugates} (or $v$ is a \emph{conjugate} of $u$) if there exist non-empty words $x$, $y$ such that $u = xy$ and $v = yx$. If $\ham(u,v)=\ham(xy,yx)=0$, then $x$ and $y$ are said to \emph{commute}. If $x$ and $y$ are both non-empty, then $v$ is said to be a \emph{non-trivial} conjugate of $u$. Let $\sigma$ be the left-shift map, so that $\sigma^i(u)=yx$ where $u=xy$ and $|x|=i$, where $i$ is an integer with $0\leq i \leq |u|.$ For example, any two of the words ${\tt eat}$, ${\tt tea}$, and ${\tt ate}$ are conjugates because ${\tt eat} =\sigma({\tt tea}) = \sigma^2({\tt ate})$.

Lyndon and Sch\"{u}tzenberger~\cite{Lyndon&Schutzenberger:1962} characterized all words $x$, $y$ that commute. Alternatively, they characterized all words $u$ that have a non-trivial conjugate $v$ such that $\ham(u,v)=0$. 
\begin{theorem}[Lyndon-Sch\"{u}tzenberger~\cite{Lyndon&Schutzenberger:1962}]\label{theorem:commute}
Let $u$ be a non-empty word. Then $u=xy$ has a non-trivial conjugate $v=yx$ such that $\ham(xy,yx)=0$ if and only if there exists a word $z$, and integers $i,j\geq 1$ such that $x=z^i$, $y=z^j$, and $u=v=z^{i+j}$.
\end{theorem}

Later, Fine and Wilf~\cite{Fine&Wilf:1965} showed that one can achieve the forward implication of Theorem~\ref{theorem:commute} with a weaker hypothesis. Namely, that $xy$ and $yx$ need not be equal, but only agree on the first $|x|+|y|-\gcd(|x|,|y|)$ terms.
\begin{theorem}[Fine-Wilf~\cite{Fine&Wilf:1965}]\label{theorem:finewilf}
Let $x$ and $y$ be non-empty words. If $xy$ and $yx$ agree on a prefix of length at least $|x|+|y|-\gcd(|x|,|y|)$, then there exists a word $z$, and integers $i,j\geq 1$ such that $x=z^i$, $y=z^j$, and $xy=yx=z^{i+j}$.
\end{theorem}

Fine and Wilf also showed that the bound of $|x|+|y|-\gcd(|x|,|y|)$ is optimal, in the sense that if $xy$ and $yx$ agree only on the first $|x|+|y|-\gcd(|x|,|y|)-1$ terms, then $xy$ need not equal $yx$. They demonstrated this by constructing words $x$, $y$ of any length such that $xy$ and $yx$ agree on the first $|x|+|y|-\gcd(|x|,|y|)-1$ terms and differ at position $|x|+|y|-\gcd(|x|+|y|)$. We call pairs of words $x$, $y$ of this form \emph{Fine-Wilf pairs}.

These words have been shown to have a close relationship with the well-known \emph{finite Sturmian words}~\cite{deLuca&Mignosi:1994}. 
\begin{example} We give some examples of words that display the optimality of the Fine-Wilf result.

\noindent
Let $x=000000010000$ and $y=00000001$. Then $|x|=12$, $|y|=8$, and $\gcd(|x|,|y|)=4$.
\begin{align}
    xy &= 000000010000000{\color{red}0}0001\nonumber \\
    yx &= 000000010000000{\color{red}1}0000\nonumber
\end{align}
Let $x=010100101010$ and $y=0101001$. Then $|x|=12$, $|y|=7$, and $\gcd(|x|,|y|)=1$.
\begin{align}
    xy &= 01010010101001010{\color{red}0}1 \nonumber \\
    yx &= 01010010101001010{\color{red}1}0\nonumber
\end{align}
\end{example}

One remarkable property of these words is that they ``almost" commute, in the sense that $xy$ and $yx$ agree for as long a prefix as possible and differ in as few positions as possible. See Lemma~\ref{lemma:finewilf2} for a proof of this property.

One might na\"{i}vely think that the smallest possible Hamming distance between $xy$ and $yx$ after $0$ is $1$, but this is incorrect. Shallit~\cite{Shallit:2009} showed that $\ham(xy,yx)\neq 1$ for any words $x$ and $y$; see Lemma~\ref{lemma:hammingNon1}. Thus, after $0$, the smallest possible Hamming distance between $xy$ and $yx$ is $2$. If $\ham(xy,yx)=2$, then we say $x$ and $y$ \emph{almost commute}.

\begin{lemma}[Shallit~\cite{Shallit:2009}]
\label{lemma:hammingNon1}
Let $x$ and $y$ be words. Then $\ham(xy,yx)\neq 1$.
\end{lemma}

A similar concept, called the \emph{$2$-error border}, was introduced in a paper by Klav\v{z}ar and Shpectorov~\cite{Sanda&Sergey:2012}. A word $w$ is said to have a \emph{$2$-error border} of length $i$ if there exists a length-$i$ prefix $u$ of $w$, and a length-$i$ suffix $u'$ of $w$ such that $w = ux = yu'$ and $\ham(u,u')=2$ for some $x$, $y$. The $2$-error border was originally introduced in an attempt to construct graphs that have properties similar to $n$-dimensional hypercubes. The $n$-dimensional hypercube is a graph that models Hamming distance between length-$n$ binary words. See~\cite{Wei:2017, Wei&Yang&Zhu:2019, Beal&Crochemore:2021} for more on $2$-error borders.

In this paper, we characterize and count all words $u$ that have a conjugate $v$ such that $\ham(u,v)=2$. As a result, we also characterize and count all pairs of words $x$, $y$ that almost commute.

Let $n$ and $i$ be integers such that $n>i\geq 1$. Let $H(n)$ denote the set of length-$n$ words $u$ over $\Sigma_k$ that have a conjugate $v$ such that $\ham(u,v) =2$. Let $h(n)=|H(n)|$. Let $H(n,i)$ denote the set of length-$n$ words $u$ over $\Sigma_k$ such that $\ham(u,\sigma^i(u))=2$. Let $h(n,i) = |H(n,i)|$.

The rest of the paper is structured as follows. In Section~\ref{section:finewilf} we prove that Fine-Wilf pairs almost commute. In Section~\ref{section:countH} we characterize the words in $H(n,i)$ and present a formula to calculate $h(n,i)$. In Section~\ref{section:interlude} we prove some properties of $H(n,i)$ and $H(n)$ that we make use of in later sections. In Section~\ref{section:finalcount} we present a formula to calculate $h(n)$. In Section~\ref{section:exactly} we count the number of length-$n$ words $u$ with \emph{exactly} one conjugate such that $\ham(u,v)=2$. In Section~\ref{section:lyndon} we count the number of Lyndon words in $H(n)$. Finally, in Section~\ref{section:asymptotics} we show that $h(n)$ grows erratically.

\section{Fine-Wilf pairs almost commute}\label{section:finewilf}
In this section we prove that Fine-Wilf pairs almost commute. This result appears without proof in~\cite{Shallit:2015}.

\begin{lemma}\label{lemma:finewilf2}
Let $x$ and $y$ be non-empty words. Suppose $xy$ and $yx$ agree on a prefix of length $|x|+|y|-\gcd(|x|,|y|)-1$ but disagree at position $|x|+|y|-\gcd(|x|,|y|)$. Then $\ham(xy,yx)=2$.
\end{lemma}
\begin{proof}
The proof is by induction on $|x|+|y|$. Suppose $xy$ and $yx$ agree on a prefix of length $|x|+|y|-\gcd(|x|,|y|)-1$ but disagree at position $|x|+|y|-\gcd(|x|,|y|)$. Without loss of generality, let $|x|\leq |y|$.

First, we take care of the case when $|x|=|y|$, which also takes care of the base case $|x|+|y|=2$. Since $|x|=|y|$, we have that $\gcd(|x|,|y|)=|x|=|y|$. Therefore, $x$ and $y$ share a prefix of length $|x|+|y|-\gcd(|x|,|y|)-1=|x|-1$ but disagree at position $|x|$. This implies that $\ham(x,y)=1$. Thus $\ham(xy,yx)=2\ham(x,y)=2$.

Suppose $|x|<|y|$. Then $\gcd(|x|,|y|) \leq |x|$. So $|x|+|y|-\gcd(|x|,|y|)-1\geq |y|-1$. Thus $xy$ and $yx$ must share a prefix of length $\geq |y|-1$. However, since $|x|<|y|$, we have that $x$ must then be a proper prefix of $y$. So write $y=xt$ for some non-empty word $t$. Then $\ham(xy,yx)=\ham(xxt,xtx)=\ham(xt,tx)$. Since $xt$, $tx$ are suffixes of $xy$, $yx$ we have that $xt$ and $tx$ agree on the first $|y|-\gcd(|x|,|y|)-1$ terms and disagree at position $|y|-\gcd(|x|,|y|)$.  Clearly $\gcd(|x|,|y|)=\gcd(|x|,|xt|) =\gcd(|x|,|x|+|t|)=\gcd(|x|,|t|)$, and $|y|-\gcd(|x|,|y|) = |x|+|t|-\gcd(|x|,|t|)$. Therefore $xt$ and $tx$ share a prefix of length $|x|+|t|-\gcd(|x|,|t|)-1$ and differ at position $|x|+|t|-\gcd(|x|,|t|)$. By induction $\ham(xt,tx)=2$, and thus $\ham(xy,yx)=2$.
\end{proof}

\section{Counting $H(n,i)$}\label{section:countH}
In this section we characterize the words in $H(n,i)$ and use this characterization to provide an explicit formula for $H(n,i)$.

\begin{lemma}\label{lemma:bijection}
Let $n$, $i$ be positive integers such that $n>i$. Let $g=\gcd(n,i)$. Let $w$ be a length-$n$ word. Let $w=x_0x_1\cdots x_{n/g-1}$ where $|x_j|=g$ for all $j$, $0\leq j \leq n/g-1$. Then $w\in H(n,i)$ iff there exist two distinct integers $j_1$, $j_2$, $0\leq j_1 < j_2\leq n/g-1$ such that $\ham(x_{j_1},x_{j_2})=1$ and $x_{j}= x_{(j+i/g)\bmod{n/g}}$ for all $j\neq j_1,j_2$, $0\leq j \leq n/g-1$.
\end{lemma}
\begin{proof} We write $w= x_0x_1\cdots x_{n/g-1}$ where $|x_j|=g$ for all $j$, $0\leq j \leq n/g-1$. Since $g$ divides $i$, we have that $\sigma^i(w) = x_{i/g}\cdots x_{n/g-1}x_0\cdots x_{i/g-1}$.

$\Longrightarrow:$ Suppose $w\in H(n,i)$. Then
\begin{align}
    \ham(w,\sigma^i(w)) &= \ham(x_0x_1\cdots x_{n/g-1},x_{i/g}\cdots x_{n/g-1}x_0\cdots x_{i/g-1})\nonumber \\
                        &= \sum_{j=0}^{n/g-1}\ham(x_j,x_{(j+i/g)\bmod{n/g}})\nonumber \\
                        &=2.\nonumber
\end{align}
In order for the Hamming distance between $w$ and $\sigma^i(w)$ to be $2$, we must have that either
\begin{itemize}
    \item $\ham(x_j,x_{(j+i/g)\bmod {n/g}})=2$ for exactly one $j$, $0 \leq j \leq n/g-1$; or
    \item $\ham(x_{j_1}, x_{(j_1+i/g)\bmod{n/g}})=1$ and $\ham(x_{j_2}, x_{(j_2+i/g)\bmod{n/g}})=1$ for two distinct integers $j_1,j_2$, $0\leq j_1<j_2\leq n/g-1$.
\end{itemize}

Suppose $\ham(x_j,x_{(j+i/g)\bmod {n/g}})=2$ for some $j$, $0 \leq j \leq n/g-1$. Then it follows that $x_p = x_{(p+i/g)\bmod {n/g}}$ for all $p\neq j$, $0\leq p \leq n/g-1$. Since $g=\gcd(n,i)$, we have that $\gcd(n/g,i/g)=1$. The additive order of $i/g$ modulo $n/g$ is $\frac{n/g}{\gcd(n/g,i/g)}=n/g$. Therefore, we have that \[x_{(j+i/g)\bmod {n/g}} = x_{(j+2i/g)\bmod {n/g}}=\cdots =x_{(j+(n/g-1)i/g)\bmod {n/g}}=x_j\] and $\ham(x_j,x_{(j+i/g)\bmod {n/g}})=2$, a contradiction.

Suppose $\ham(x_{j_1}, x_{(j_1+i/g)\bmod{n/g}})=1$ and $\ham(x_{j_2}, x_{(j_2+i/g)\bmod{n/g}})=1$ for two distinct integers $j_1,j_2$, $0\leq j_1<j_2\leq n/g-1$. Then it follows that $x_j = x_{(j+i/g)\bmod{n/g}}$ for all $j \neq j_1,j_2$, $0\leq j \leq n/g-1$. Since the additive order of $i/g$ modulo $n/g$ is $n/g$, we have that if we start at $j_1$ and successively add $i/g$ and    take the result modulo $n/g$, then we will reach every integer between $0$ and $n/g-1$. Therefore, we will reach $j_2$ before we reach $j_1$ again. Thus, since $x_j=x_{(j+i/g)\bmod{n/g}}$ for all $j\neq j_1,j_2$, $0\leq j \leq n/g-1$, we have that \[x_{(j_1+i/g)\bmod{n/g}} = x_{(j_1+2i/g)\bmod{n/g}} = \cdots = x_{j_2}.\] But now we have $\ham(x_{j_1}, x_{(j_1+i/g)\bmod{n/g}})=1$ and $x_{(j_1+i/g)\bmod{n/g}} =x_{j_2}$, which implies $\ham(x_{j_1},x_{j_2})=1$.

$\Longleftarrow:$ Suppose there exist two distinct integers $j_1$, $j_2$, $0\leq j_1<j_2\leq n/g-1$ such that $\ham(x_{j_1},x_{j_2}) = 1$ and $x_j=x_{(j+i/g)\bmod{n/g}}$ for all $j\neq j_1,j_2$, $0\leq j \leq n/g-1$. Since the additive order of $i/g$ modulo $n/g$ is $n/g$, we have that if we start at $j_1$ and successively add $i/g$ modulo $n/g$, then we will reach every integer between $0$ and $n/g-1$. But this means that we will reach $j_2$ before we get to $j_1$ again. Thus, we have that \[x_{(j_1+i/g)\bmod{n/g}} = x_{(j_1+2i/g)\bmod{n/g}} = \cdots = x_{j_2}.\] Similarly, if we start at $j_2$ and successively add $i/g$ modulo $n/g$ we will reach $j_1$ before looping back to $j_2$. So \[x_{(j_2+i/g)\bmod{n/g}} = x_{(j_2+2i/g)\bmod{n/g}} = \cdots = x_{j_1}.\] Therefore, we have that $w\in H(n,i)$ since

\begin{align}
    \ham(w,\sigma^{i}(w)) &= \ham(x_0x_1\cdots x_{n/g-1},x_{i/g}\cdots x_{n/g-1}x_0\cdots x_{i/g-1})\nonumber \\
    &=\sum_{j=0}^{n/g-1}\ham(x_j,x_{(j+i/g)\bmod{n/g}})\nonumber \\
    &= \ham(x_{j_1}, x_{(j_1+i/g)\bmod{n/g}}) +\ham(x_{j_2}, x_{(j_2+i/g)\bmod{n/g}})\nonumber\\
    &= \ham(x_{j_1}, x_{j_2}) +\ham(x_{j_2}, x_{j_1})\nonumber\\
    &=2.\nonumber
\end{align}
\end{proof}

\begin{lemma}\label{lemma:formula} Let $n,i$ be positive integers such that $n>i.$ Then
\[h(n,i) = \frac{1}{2}k^{\gcd(n,i)}(k-1)n\bigg(\frac{n}{\gcd(n,i)}-1\bigg).\]
\end{lemma}
\begin{proof}
Let $w$ be a length-$n$ word. Let $g=\gcd(n,i)$. We split up $w$ into length-$g$ blocks. We write $w=x_0x_1\cdots x_{n/g-1}$ where $|x_j|=g$ for all $j$, $0\leq j \leq n/g-1$. Lemma~\ref{lemma:bijection} gives a complete characterization of $H(n,i)$. Namely, the word $w$ is in $H(n,i)$ if and only if there exist two distinct integers $j_1,j_2$, $0\leq j_1<j_2\leq n/g-1$ such that $\ham(x_{j_1},x_{j_2})=1$ and $x_j=x_{(j+i/g)\bmod{n/g}}$ for all $j\neq j_1,j_2$, $0\leq j\leq n/g-1$. Given $j_1$, $j_2$, $x_{j_1}$, and $x_{j_2}$, all $x_j$ for $j\neq j_1,j_2$, $0\leq j \leq n/g-1$ are already determined.

There are \[\sum_{j_2=1}^{n/g-1}\sum_{j_1=0}^{j_2-1} 1 = \frac{1}{2}\frac{n}{g}\bigg(\frac{n}{g}-1\bigg)\] choices for $j_1$ and $j_2$. There are $k^g$ options for $x_{j_1}$. Considering that $x_{j_1}$ and $x_{j_2}$ differ in exactly one position, there are $g(k-1)$ choices for $x_{j_2}$ given $x_{j_1}$. Putting everything together we have that
\begin{align}
    h(n,i) &= 
    \overbrace{\frac{1}{2}\frac{n}{g}\bigg(\frac{n}{g} -1\bigg)}^{\text{choices for }j_1\text{ and } j_2}\overbrace{k^g}^{\text{choices for }x_{j_1}}\overbrace{g(k-1)}^{\text{choices for }x_{j_2}\text{ given }x_{j_1}}\nonumber\\
    &= \frac{1}{2}k^{\gcd(n,i)}(k-1)n\bigg(\frac{n}{\gcd(n,i)}-1\bigg). \nonumber 
\end{align}
\end{proof}

\begin{corollary}
Let $m,n\geq 1$ be integers. Then there are exactly \[h(n+m,m) = \frac{1}{2}k^{\gcd(n+m,m)}(k-1)(n+m)\bigg(\frac{n+m}{\gcd(n+m,m)}-1\bigg).\] pairs of words $(x,y)$ of length $(m,n)$ such that $\ham(xy,yx)=2$.
\end{corollary}

\section{Some useful properties}\label{section:interlude}

In this section we prove some properties of $H(n,i)$ and $H(n)$ that we use in later sections.

\begin{lemma}\label{lemma:half}
Let $u$ be a length-$n$ word. Let $i$ be an integer with $0< i < n$. If $u\in H(n,i)$ then $u\in H(n,n-i).$ 
\end{lemma}
\begin{proof}
Suppose $i\leq n/2$. Then we can write $u=xtz$ for some words $t,z$ where $|x|=|z|=i$ and $|t|=n-2i$. We have that $\ham(xtz,tzx) = \ham(xt,tz) + \ham(z,x) = 2$. Consider the word $zxt$. Clearly $v=zxt$ is a conjugate of $u=xtz$ such that $\ham(xtz,zxt) = \ham(x,z) +\ham(tz,xt) = 2$ where $u = (xt)z$ and $v = z(xt)$ with $|xt| = n-i$. Therefore $u\in H(n,n-i).$

Suppose $i>n/2$. Then we can write $u=zty$ for some words $t,z$ where $|z|=|y|=n-i$ and $|t|=2i-n$. We have that $\ham(zty,yzt) = \ham(z,y) + \ham(ty,zt) = 2$. Consider the word $tyz.$ Clearly $v=tyz$ is a conjugate of $u=zty$ such that $\ham(zty,tyz) = \ham(zt,ty) +\ham(y,z) = 2$ where $u = z(ty)$ and $v = (ty)z$ with $|z| =n- i$. Therefore $u\in H(n,n-i).$
\end{proof}

\begin{lemma}
Let $u$ be a length-$n$ word. If $u\in H(n)$, then $\ham(u,v)>0$ for any non-trivial conjugate $v$ of $u$.
\end{lemma}
\begin{proof}
We prove the contrapositive of the lemma statement. Namely, we prove that if there exists a non-trivial conjugate $v$ of $u$ such that $\ham(u,v)=0$ then $u\not \in H(n)$. 

Suppose $u=xy$ and $v=yx$ for some non-empty words $x$, $y$. Then by Theorem~\ref{theorem:commute} we have that there exists a word $z$, and an integer $i\geq 2$ such that $u = v = z^i$. Let $w$ be a conjugate of $u$. Then $w = (ts)^i$ where $z=st$. So $\ham(u,w) = \ham((st)^i,(ts)^i) = i\ham(st,ts)$. If $st=ts$, then $\ham(u,w)=0$. If $st\neq ts$, then $\ham(st,ts)\geq 2$ (Lemma~\ref{lemma:hammingNon1}). Since $\ham(st,ts)\geq 2$ and $i\geq 2$, we have $\ham(u,w)\geq 4$. Thus $u\not\in H(n)$.
\end{proof}
\begin{corollary}\label{corollary:notH}
Let $u$ be a length-$n$ word. If $u$ is a power, then $u\not\in H(n)$.
\end{corollary}
\begin{corollary}\label{corollary:primitive}
All words in $H(n)$ are primitive.
\end{corollary}

\begin{lemma}\label{lemma:anyconjugate}
Let $u$ be a length-$n$ word. Let $i$ be an integer with $0< i < n$. If $u\in H(n,i)$, then any conjugate of $u$ is also in $H(n,i)$.
\end{lemma}
\begin{proof}
Suppose $u\in H(n,i)$. Then $\ham(u,\sigma^i(u))=2$. If we shift both $u$ and $\sigma^i(u)$ by the same amount, then the symbols that are being compared to each other do not change. Thus $\ham(\sigma^j(u), \sigma^{i+j}(u)) = 2$ for all $j\geq 0$. So any conjugate $\sigma^j(u)$ of $u$ must also be in $H(n,i)$.
\end{proof}
\section{Counting $H(n)$}\label{section:finalcount}

Lemma~\ref{lemma:half} shows that $H(n,i) = H(n,n-i)$, which in turn implies that $h(n) \leq \sum_{i=1}^{\lfloor n/2\rfloor} h(n,i)$. To make this inequality an equality we need to be able to account for those words that are double-counted in the sum $\sum_{i=1}^{\lfloor n/2\rfloor} h(n,i)$. In this section we resolve this problem and give an exact formula for $h(n)$. More specifically, we show that all words $w$ that are in both $H(n,i)$ and $H(n,j)$, for $i\neq j$, must exhibit a certain regular structure that we can explicitly describe. Then we use this structure result, in addition to the results from Section~\ref{section:countH} and Section~\ref{section:interlude}, to give an exact formula for $h(n)$.


\begin{lemma}\label{lemma:twoPeriods}
Let $n,i,j$ be positive integers such that $n\geq 2i>2j$. Let $g=\gcd(n,i,j)$. Let $w$ be a length-$n$ word. Then $w\in H(n,i)$ and $w\in H(n,j)$ if and only if there exists a word $u$ of length $g$, a word $v$ of length $g$ with $\ham(u,v)=1$, and a non-negative integer $p<n/g$ such that $w=u^pvu^{n/g-p-1}$.
\end{lemma}
\begin{proof}~\\
\noindent $\Longrightarrow:$ The proof is by induction on $|w|=n$. Suppose $w\in H(n,i)$ and $w\in H(n,j)$. First, we take care of the case when $n=2i$, which also includes the base case $n=4$, $i=2$, $j=1$. Write $w = xyx'y'$ where $|xy|=|x'y'|=i=n/2$ and $|x|=|x'|=j$. Since $w\in H(n,i)$, we have that $\ham(xyx'y',x'y'xy)=2$. This implies that $\ham(xy,x'y')=1$. Furthermore, if $\ham(xy,x'y')=1$ then either $\ham(x,x')=1$ or $\ham(y,y')=1$.

Suppose $\ham(x,x')=1$. Then $y=y'$. Since $w\in H(n,j)$, we have $\ham(xyx'y,yx'yx)=\ham(xy,yx')+\ham(x'y,yx) =2$. Suppose $\ham(xy,yx')=0$ or $\ham(x'y,yx)=0$. Both cases imply that $\ham(xy,yx)=1$, which contradicts Lemma~\ref{lemma:hammingNon1}. Thus, we must have $\ham(xy,yx')=\ham(x'y,yx)=1$. But this implies that $\ham(xy,yx)=0$ and $\ham(x'y,yx')=2$ or $\ham(xy,yx)=2$ and $\ham(x'y,yx')=0$. Without loss of generality, suppose $\ham(xy,yx)=0$. By Theorem~\ref{theorem:commute}, there exists a word $s$, and integers $l,m\geq 1$ such that $x=s^l$ and $y=s^m$. Clearly $|s|$ divides $\gcd(n/2, j) =\gcd(n,n/2,j)= \gcd(n,i,j) = g$ since it divides both $|x| = j$ and $|xy| =i= n/2$. Therefore, there exists a length-$g$ word $u$ such that $x = u^{j/g}$ and $y = u^{(i-j)/g}$. Since $x$ and $x'$ differ in exactly one position, and $x=u^{j/g}$, there exists a length-$g$ word $v$ with $\ham(u,v)=1$, and a non-negative integer $p'<j/g$ such that $x' = u^{p'}vu^{j/g-p'-1}$. Letting $p = p' + i/g = p' + (n/2)/g$, we have $w=xyx'y = u^{i/g}u^{p'}vu^{j/g-p'-1}u^{(i-j)/g} = u^pvu^{n/g-p-1}$.

Suppose $\ham(y,y')=1$. Then $x=x'$. Since $w\in H(n,j)$, we have $\ham(xyxy',yxy'x) = \ham(xy,yx) + \ham(xy',y'x)=2$. By Lemma~\ref{lemma:hammingNon1}, we have that $\ham(xy,yx)\neq 1$ and $\ham(xy',y'x)\neq 1$. So either $\ham(xy,yx)=0$ or $\ham(xy',y'x)=0$. Without loss of generality, suppose $\ham(xy,yx)=0$. As in the previous case when $\ham(x,x')=1$, there exists a length-$g$ word $u$ such that $x = u^{j/g}$ and $y = u^{(i-j)/g}$. Since $y$ and $y'$ differ in exactly one position, there exists a length-$g$ word $v$ with $\ham(u,v)=1$, and a non-negative integer $p' < (i-j)/g$ such that $y' = u^{p'} v u^{(i-j)/g - p' - 1}$. Letting $p = p' + (i+j)/g = p' + (n/2+j)/g$, we have $w = xyxy' = u^{i/g}u^{j/g}u^{p'} v u^{(i-j)/g - p' - 1} = u^p v u^{n/g-p-1}$.

Now, we take care of the case when $n > 2i$. Write $w=xyx'y'z$ for words $x,y,x',y',z$ where $|xy|=|x'y'|=i$, and $|x|=|x'|=j$. Since $w\in H(n,i)$, we have that $w$ and $\sigma^i(w)$ differ in exactly two positions $j_1 < j_2$. But $n>2i$ implies that either $j_2-j_1 > i$ or $j_2-j_1\leq i$ and $n-(j_2-j_1) > 2i - (j_2-j_1) \geq i$. In either case we have that there is a length-$i$ contiguous block, possibly occurring in the wraparound, where $w$ and $\sigma^i(w)$ match. This translates to there being a length-$2i$ block in $w$ of the form $tt$ where $|t|=i$. Additionally, we have that $\sigma^m(w)\in H(n,i)$ and $\sigma^m(w)\in H(n,j)$ for all $m\geq 0$ by Lemma~\ref{lemma:anyconjugate}. Therefore, we can assume without loss of generality that $w$ begins with this length-$2i$ block (i.e., $\ham(xy,x'y')=0$).

Suppose $\ham(xy,x'y')=0$. Then $\ham(xyxyz,xyzxy)=\ham(xyxyz,yxyzx)=2$. Clearly $\ham(xyxyz,xyzxy)=\ham(xyz,zxy)=2$, so $xyz \in H(n-i,i)$. Now, either $xy=yx$ or $xy\neq yx$. If $xy=yx$, then we clearly have $\ham(xyxyz,yxyzx) = \ham(xyz,yzx)=2$. Therefore, we have $xyz\in H(n-i,j)$. Let $g = \gcd(n-i,i,j)$. We have that $g=\gcd(n-i,i,j) = \gcd(\gcd(n-i,i),j) = \gcd(\gcd(n,i),j) =\gcd(n,i,j)$. If $n-i \geq 2i> 2j$, then we can apply induction to $xyz$ directly. By Lemma~\ref{lemma:half}, we have that if $xyz \in H(n-i,i)$ and $xyz\in H(n-i,j)$, then $xyz\in H(n-i,n-2i)$ and $xyz\in H(n-i,n-i-j)$. If $n-i < 2i$ and $n-i \geq 2j$, then $n-i > 2(n-2i)$ and $\gcd(n-i,n-2i,j) = \gcd(n,i,j)=g$. However, in this case we can have $j=n-2i$, which we have to take care of separately since it does not satisfy the inductive hypothesis. If $n-i < 2j < 2i$, then $n-i > 2(n-i-j)$,  $n-i > 2(n-2i)$, and $\gcd(n-i,n-2i,n-i-j) = \gcd(n,i,j)=g$.

Suppose $j\neq n-2i$. By induction there exists a word $u$ of length $g$, a word $v$ of length $g$ with $\ham(u,v)=1$, and a non-negative integer $p' < (n-i)/g$ such that $xyz=u^{p'}vu^{(n-i)/g-p'-1}$. Since $xy=yx$ and $g\mid \gcd(i,j)$, it is clear that $xy = u^{i/g}$. Then $w = xyxyz = u^{p'+i/g} v u^{(n-i)/g-p'-1}$. Letting $p=p'+i/g$, we have $w  = u^{p} v u^{n/g-p-1}$.

Suppose $j=n-2i$. Then $w=xyxyz$ where $|z|=|x|=n-2i$. Since $w\in H(n,n-2i)$, we have $\ham(xyxyz,yxyzx) = \ham(xy,yx)+\ham(xy,yz) +\ham(z,x)=2$. But $xy=yx$ by assumption. Thus $\ham(xy,yz)+\ham(z,x)=2$, which is only true when $\ham(z,x)=1$. By Theorem~\ref{theorem:commute}, there exists a word $s$, and integers $l,m\geq 1$ such that $x=s^l$ and $y=s^m$. Since $|s|$ divides both $|x|=j=n-2i$ and $|xy|=i$, we have $|s|$ divides $\gcd(i,j)=\gcd(i,n-2i) = \gcd(n, i, n- 2i)=\gcd(n,i,j) = g$. Therefore, there exists a length-$g$ word $u$ such that $x=u^{j/g}$ and $y=u^{(i-j)/g}$. We also have $\ham(z,x)=1$, which implies that there exists a length-$g$ word $v$ with $\ham(u,v)=1$, and a non-negative integer $p' < j/g$ such that $z = u^{p'}vu^{j/g-p'-1}$. Letting $p=p'+2i/g$, we have $w=xyxyz=u^{2i/g}u^{p'}vu^{(n-2i)/g-p'-1}=u^{p}vu^{n/g-p-1}$.

If $xy\neq yx$, then we must have $\ham(xy,yx)=2$. But since $\ham(xyxyz,yxyzx) = 2$, we must have $\ham(xyz,yzx) = 0$. This means that $xyz$ is a power, but we have already demonstrated that $xyz \in H(n-i,i)$. By Corollary~\ref{corollary:notH}, this is a contradiction.

\bigskip

\noindent $\Longleftarrow:$ Let $g=\gcd(n,i,j)$. Suppose we can write $w = u^{p}vu^{n/g-p-1}$ where $|u|=|v|=g$, and $\ham(u,v)=1$. Since $g\mid i$, we can write

\[\ham(w,\sigma^i(w)) =  \ham(u^{p}vu^{n/g-p-1}, u^{p-i/g}vu^{n/g+i/g-p-1}) = 2\ham(u,v)=2\]
if $p\leq i/g$, and
\[\ham(w,\sigma^i(w)) =  \ham(u^{p}vu^{n/g-p-1}, u^{n/g  -i+p}v u^{p-i-1}) = 2\ham(u,v)=2\]
if $p> i/g$.
Since $g$ divides $j$ as well, a similar argument works to show $\ham(w,\sigma^{j}(w))=2$ as well. Therefore, $w\in H(n,i)$ and $w\in H(n,j)$.
\end{proof}


Lemma~\ref{lemma:twoPeriods} shows that any word $w$ that is in $H(n,i)$ and $H(n,j)$ for $j<i\leq n/2$ is of Hamming distance $1$ away from a power. Therefore, to count the number of such words, we need a formula for the number of powers.

 Clearly a word is a power if and only if it is not primitive. This implies that $p_k(n)=k^n-\psi_k(n)$ where $\psi_k(n)$ is the number of length-$n$ primitive words over a $k$-letter alphabet. From Lothaire's 1983 book~\cite[p.~9]{Lothaire:1983} we also have that 

 \[\psi_k(n) = \sum_{d\mid n}\mu(d)k^{n/d}\] where $\mu$ is the M\"{o}bius function.

Let $H'(n,i)$ denote the set of words $w\in H(n,i)$ that are also in $H(n,j)$ for some $j<i$. Let $h'(n,i) = |H'(n,i)|$.

\begin{corollary}
Let $n,i$ be positive integers such that $n\geq 2i$. Then
\[h'(n,i) = \begin{cases} 
      n(k-1)p_k(i), & \text{if $i\mid n$;} \\
      n(k-1)k^{\gcd(n,i)}, & \text{otherwise.} 
   \end{cases}\]
\end{corollary}

Let $H''(n,i)$ denote the set of words $w\in H(n,i)$ such that $w\not\in H(n,j)$ for all $j<i$. Let $h''(n,i) = |H''(n,i)|$.

\begin{lemma}\label{lemma:shortest}
Let $n,i$ be positive integers such that $n >i$. Then
\[h''(n,i) = \begin{cases} 
      \frac{1}{2}n(k-1)\big(k^{\gcd(n,i)}\big(\frac{n}{\gcd(n,i)}-1\big)-2p_k(i)\big), & \text{if $i\mid n$;} \\
      \frac{1}{2}k^{\gcd(n,i)}(k-1)n\big(\frac{n}{\gcd(n,i)}-3\big), & \text{otherwise.} 
   \end{cases}\]
\end{lemma}
\begin{proof}
Let $w$ be a length-$n$ word. The word $w$ is in $H''(n,i)$ precisely if it is in $H(n,i)$ but not in any $H(n,j)$ for $j<i$. So computing $h''(n,i)$ reduces to computing the number of length-$n$ words that are in $H(n,i)$ and $H(n,j)$ for some $j<i$ (i.e., $h'(n,i)$) and then subtracting it from the number of words in $H(n,i)$ (i.e., $h(n,i)$). Therefore  
\[h''(n,i) =h(n,i)-h'(n,i)= \begin{cases} 
      \frac{1}{2}n(k-1)\big(k^{\gcd(n,i)}\big(\frac{n}{\gcd(n,i)}-1\big)-2p_k(i)\big), & \text{if $i\mid n$;} \\
      \frac{1}{2}k^{\gcd(n,i)}(k-1)n\big(\frac{n}{\gcd(n,i)}-3\big), & \text{otherwise.} 
   \end{cases}\]
\end{proof}

\begin{theorem}
Let $n$ be an integer $\geq 2$. Then \[h(n) = \sum_{i=1}^{\lfloor n/2\rfloor}h''(n,i).\]
\end{theorem}
\begin{proof}
Every word that is in $H(n)$ must also be in $H(n,i)$ for some integer $i$ in the range $1\leq i \leq n-1$. By Lemma~\ref{lemma:half} we have that every word that is in $H(n,i)$ is also in $H(n,n-i)$. Therefore we only need to consider words in $H(n,i)$ where $i$ is an integer with $i\leq n-i \implies i \leq n/2.$ Consider the quantity $S= \sum\limits_{i=1}^{\lfloor n/2 \rfloor} h(n,i).$ Since any member of $H(n)$ must also be a member of $H(n,i)$ for some $i\leq \lfloor n/2\rfloor$, we have that $h(n) \leq S.$ But any member of $H(n,i)$ may also be a member of $H(n,j)$ for some $j<i.$ These words are accounted for multiple times in the sum $S$. To avoid double-counting we must count the number of words $w$ that are in $H(n,i)$ but not in $H(n,j)$ for any $j<i$. This quantity is exactly $h''(n,i)$. Therefore \[h(n) = \sum_{i=1}^{\lfloor n/2\rfloor} h''(n,i).\] \end{proof}

\section{Exactly one conjugate}\label{section:exactly}
So far we have been interested in length-$n$ words $u$ that have at least one conjugate of Hamming distance $2$ away from $u$. But what about length-$n$ words $u$ that have exactly one conjugate of Hamming distance $2$ away from $u$? In this section we provide a formula for the number $h'''(n)$ of length-$n$ words $u$ with exactly one conjugate $v$ such that $\ham(u,v)=2$. 

Let $n$ and $i$ be positive integers such that $n > i$. Let $H'''(n)$ denote the set of length-$n$ words $u$ over $\Sigma_k$ that have exactly one conjugate $v$ with $\ham(u,v)=2$. Let $h'''(n) = |H'''(n)|$. Let $H''(n,i)$ denote the set of length-$n$ words $w$ such that $w$ is in $H(n,i)$ but is not in $H(n,j)$ for any $j\neq i.$ Let $h''(n,i) = |H''(n,i)|$.

 Suppose $w\in H'''(n,i)$. Then by definition we have that $w\in H(n,i)$ and $w\not\in H(n,j)$ for any $j\neq i$. But by Lemma~\ref{lemma:half} we have that if $w$ is in $H(n,i)$ then it must also be in $H(n,n-i)$. So if $i \neq n-i$, then $w$ has at least two distinct conjugates of Hamming distance $2$ away from it, namely $\sigma^i(w)$ and $\sigma^{n-i}(w)$. Therefore we have $i=n-i$. This implies that $n$ must be even, so $H'''(2m+1)=\{\}$ for all $m\geq 1.$ Since $i=n-i\implies i=n/2$, we have that $w\in H(n,n/2)$. However $w$ cannot be in $H(n,j)$ for any $j\neq n/2$. Since any word in $H(n,j)$ is also in $H(n,n-j)$, the condition of $w\not \in H(n,j)$ for any $j\neq n/2$ is equivalent to $w\not\in H(n,j)$ for any $j$ with $1\leq j < n/2.$ But this is just the definition of $H''(n,n/2)$. From this we get the following theorem.
\begin{theorem}
Let $n\geq 1$ be an integer. Then \[h'''(n)= \begin{cases} 
      \frac{1}{2}n(k-1)(k^{n/2}-2p_k(n/2)), & \text{if $n$ is even;} \\
     0, & \text{otherwise.} 
   \end{cases}\]
\end{theorem}

\section{Lyndon conjugates}\label{section:lyndon}
A \emph{Lyndon word} is a word that is lexicographically smaller than any of its non-trivial conjugates. In this section we count the number of Lyndon words in $H(n)$. 

\begin{theorem}\label{theorem:lyndon}
There are $\frac{h(n)}{n}$ Lyndon words in $H(n)$.
\end{theorem}
\begin{proof}
Corollary~\ref{corollary:primitive} says that all members of $H(n)$ are primitive and Lemma~\ref{lemma:anyconjugate} says that if a word is in $H(n)$, then any conjugate of it is also in $H(n)$. It is easy to verify that every primitive word has exactly one Lyndon conjugate. Therefore exactly $\frac{h(n)}{n}$ words in $H(n)$ are Lyndon words.
\end{proof}
\section{Asymptotic behaviour of $h(n)$}\label{section:asymptotics}
In this section we show that $h(n)$ grows erratically. We do this by demonstrating that $h(n)$ is a cubic polynomial for prime $n$, and that $h(n)$ is bounded below by an exponential for even $n$. 
\begin{lemma}
Let $n$ be a prime number. Then \[h(n)=\frac{1}{4}k(k-1)n(n^2-4n+7).\]
\end{lemma}
\begin{proof}
Let $n>1$ be a prime number. Since $n$ is prime, we have that $\gcd(n,i)=1$ for all integers $i$ with $1<i<n$. Then
\begin{align} h(n) &= \sum_{i=1}^{(n-1)/2}h''(n,i) \nonumber\\
                   &=\frac{1}{2}k(k-1)n(n-1)+\sum_{i=2}^{(n-1)/2}\frac{1}{2}k^{\gcd(n,i)}(k-1)n\bigg(\frac{n}{\gcd(n,i)}-3\bigg)\nonumber \\
                   &= \frac{1}{2}k(k-1)n(n-1)+\bigg(\frac{n-3}{2}\bigg)\frac{1}{2}k(k-1)n(n-3) \nonumber \\
                   &= \frac{1}{4}k(k-1)n(n^2-4n+7).\nonumber
\end{align}
\end{proof}
\begin{lemma}
Let $n> 1$ be an integer. Then $h(2n) \geq nk^{n}$.
\end{lemma}
\begin{proof}
Since any word in $H(2n,n)$ must also be in $H(2n)$, we have that $h(2n) \geq h(2n,n)$. From Lemma~\ref{lemma:formula} we see that $h(2n,n) = \frac{1}{2}k^{\gcd(2n,n)}(k-1)2n\big(\frac{2n}{\gcd(2n,n)}-1\big) =k^n (k-1)n.$ Since $k\geq 2$, we have that $k-1\geq 1$. Therefore $h(2n) \geq k^n (k-1)n \geq  nk^n$ for all $n>1$.
\end{proof}

\bibliographystyle{unsrt}
\bibliography{abbrevs,main}

\end{document}